\documentclass[12pt]{amsart}
\usepackage[dvips]{graphicx}
\usepackage{amsmath,graphics}
\usepackage{amsfonts,amssymb}
\usepackage{xypic}
\usepackage{comment}
\usepackage{color}
\usepackage{stmaryrd}
\specialcomment{proofc}{}{}
\includecomment{proofc}
\theoremstyle{plain}
\newtheorem*{theorem*}{Theorem}
\newtheorem*{lemma*} {Lemma}
\newtheorem*{corollary*} {Corollary}
\newtheorem*{proposition*}{Proposition}
\newtheorem*{conjecture*}{Conjecture}

\newtheorem*{theorem1*}{Theorem 1}
\newtheorem*{theorem2*}{Theorem 2}
\newtheorem*{theorem3*}{Theorem 3}

\theoremstyle{remark}

\newtheorem{example*}{Example}

\theoremstyle{definition}

   \def\Z{\Bbb{Z}}  
    
    \def\bp{\begin{pmatrix}}
 \def\ep{\end{pmatrix}} \def\bn{\begin{enumerate}} 
   \def\en{\end{enumerate}}
\def\ba{\begin{array}} \def\ea{\end{array}}

\def\be{\begin{equation}} \def\ee{\end{equation}}

\begin{document}
\title[Thompson's Group F is not SCY]{Thompson's Group F is not SCY}
\author{Stefan Friedl}
\address{Fakult\"at f\"ur Mathematik\\ Universit\"at Regensburg\\   Germany}
\email{sfriedl@gmail.com}

\author{Stefano Vidussi}
\address{Department of Mathematics, University of California,
Riverside, CA 92521, USA} \email{svidussi@math.ucr.edu} \thanks{S. Vidussi was partially supported by NSF grant
\#0906281.}

\begin{abstract} In this note we prove that Thompson's group $F$ cannot be the fundamental group of a symplectic $4$--manifold $M$ with canonical class $K = 0 \in H^{2}(M)$ by showing that its Hausmann--Weinberger invariant $q(F)$ is strictly positive.

\end{abstract}
\maketitle


\vspace{-0.55cm}
Symplectic $4$--manifolds with trivial canonical class, oftentimes referred to as \textit{symplectic Calabi--Yau} manifolds are, conjecturally, a fairly restricted class of manifolds, see \cite{Li06a,Do08}. Part of this restriction is reflected in known constraints for their fundamental groups, that we will refer to as \textit{SCY groups}. In the case of $b_{1} > 0$, these results, for which  we refer to \cite{Bau08,Li06b,FV13}, corroborate the expectation that such groups are  (virtually) poly--$\Z$.

We are interested here in the following constraints, that apply to the fundamental group $G = \pi_{1}(M)$ of a symplectic Calabi--Yau $4$--manifold $M$  with $b_{1}(M) = b_{1}(G) > 0$:  
\begin{enumerate} 
\item $2 \leq b_{1}(G) \leq vb_{1}(G) \leq 4$, where $vb_{1}(G) =  \mbox{sup}\{ b_{1}(G_{i}) | G_{i} \leq_{f.i.} G\}$  denotes the supremum of the first Betti number of all finite index subgroups of $G$; 
\item if the first  $L^{2}$--Betti number $b_{1}^{(2)}(G)$ vanishes, then $q(G) = 0$, where  the Hausmann--Weinberger invariant $q(G) = \mbox{inf} \{ \chi(X)| \pi_1(X) = G \}$ is defined as the infimum of the Euler characteristic among all $4$--manifolds whose fundamental group is $G$. 
\end{enumerate}
 (In \cite[Proposition 2.2]{FV13} the vanishing of $q(G)$  is stated under the assumption that $G$ is residually finite, but in fact only the condition $b_{1}^{(2)}(G) = 0$ is used in the proof.)

The purpose of this note is to apply these constraints to the case of Thompson's group $F$. The group $F$ (about which we refer to \cite{CFP96,Ge08} for some preliminary facts) is a group that admits the finite presentation \begin{equation} \label{eq:pres} F = \langle  x_{0},x_{1}|  [x_{0}x_{1}^{-1},x_{0}^{-1}x_1x_0] = 1, [x_{0}x_{1}^{-1},x_{0}^{-2}x_1x^{2}_0] = 1. \rangle \end{equation} This group has a number of peculiar features, that make it a natural testing ground for conjectures and speculations. We should mention that S. Bauer asked (\cite[Question 1.5]{Bau08}) if another of Thompson's groups, $T$ (which is a finitely presented simple group) is SCY, in this case with $b_1 = 0$: that question partly motivated the present note.

From a geometer's viewpoint, Thompson's group $F$ has already been knocked out from the royalty of groups, i.e. K\"ahler groups, by the work of \cite{NR06} (whose authors will hopefully condone us for the slight plagiarism in our title). However, as any finitely presented group, it keeps a footing as fundamental group of a symplectic $4$--manifold, by \cite{Go95} and, pushing the dimension up by $2$, of symplectic $6$--manifolds with trivial canonical class by \cite{FiPa11}. In spite of that, we will show that the constraints discussed above are sufficient to show that $F$ is not SCY. The main difficulty lies in the fact that the constraint on the first virtual Betti number, that is often very effective, is inconclusive: 
\begin{proposition*} \label{prop:virtual} Thompson's group $F$ satisfies $b_{1}(F) = vb_{1}(F) = 2$. \end{proposition*} 
\begin{proof} This is a consequence of the fact that (\cite[Theorems 4.5]{CFP96}) the commutator subgroup $[F,F]$ is simple. Indeed, let $N\trianglelefteq_f F$ be a finite index normal subgroup. Then $[N,N]$ is a normal subgroup of $[F,F]$.
Since $[F,F]$ is simple (and, as $F$ is not virtually abelian, $N$ is not abelian) it follows that  $[N,N]=[F,F]$. We therefore see that $H_1(N) = N/[N,N]=N/[F,F]$ is a subgroup of $F/[F,F]\cong \Z^2$. Now, as the Betti number is non--decreasing on finite index subgroups, $b_{1}(N) \geq b_1(F)$. This entails that $H_1(N)$ is a finite index subgroup of $H_1(F)$, hence a copy of $\Z^2$ itself. \end{proof}

As the constraint on the virtual Betti number is inconclusive, we must resort to the Hausmann-Weinberger invariant $q(F)$ (whose calculation,  to the authors' knowledge, has not appeared in the literature). While we are not able to calculate it exactly, we will show that it is strictly positive, whence $F$ is not a SCY group.

\begin{theorem*} The Hausmann--Weinberger invariant of Thompson's group $F$ satisfies $0 < q(F) \leq 2$. \end{theorem*} 
\begin{proof}  As is well--known (see e.g. \cite{E97}) the Hausmann--Weinberger invariant satisfies the basic inequalities $ 2 - 2b_{1}(F) \leq q(F) \leq 2 - 2 \mbox{def}(F)$, where $\mbox{def}(F)$ denotes the deficiency of $F$. The upper bound is easily obtained then from the fact that the presentation in (\ref{eq:pres}) has deficiency $0$. To prove the lower bound, we will argue by contradiction. To start, we will compute the  first $L^{2}$--Betti number. If $F$ were residually finite, the proposition, together with the L\"uck Approximation Theorem \cite{Lu94}, would immediately imply its vanishing, but as this isn't the case one must argue differently.  There is more than one way to proceed to this calculation (see \cite[Theorem 7.10]{Lu02} for the original calculation, or \cite[Theorem 1.8]{BFS12}). For the reader's benefit, we present the following, which is fairly explicit. Start with a well--known infinite presentation of the group $F$: \[  F =  \langle x_{0},x_{1},...|x_{n} x_{i} = x_{i} x_{n+1}, \forall \,\ 0 \leq i < n \rangle, \] that reduces to that in (\ref{eq:pres}) putting $x_{n} =  x_{0}^{1-n} x_1 x_{0}^{n-1}$ for all $n \geq 2$. Defining the shift monomorphism $\phi : F \to F$ as $\phi(x_{i}) = x_{i+1}$ for all $i \geq 0$, the images $F(m) = \phi^{m}(F)$ are isomorphic to $F$ itself, and $F$ is the properly ascending HNN-extension with base $F(1)$ itself, bonding subgroups $F(1)$ and $F(2)$ and stable letter $x_{0}$, i.e. \[ F = \langle F(1),x_{0}|x_{0}^{-1}F(1) x_{0} = \phi(F(1))\rangle. \] As $F$ (hence $F(1),F(2)$) admits a finite presentation, the $L^2$--Betti number $b_{1}^{(2)}(F)$ vanishes by \cite[Lemma 2.1]{Hi02}. As $F$ is an infinite group,  $b_{0}^{(2)}(F)$  vanishes as well.  Let $M$ be a $4$--manifold with fundamental group $F$. By standard facts of $L^2$-invariants (see e.g. \cite{Lu02}) we have \[ \chi(M) =   2 b_{0}^{(2)}(F) -2  b_{1}^{(2)}(F) +  b_{2}^{(2)}(M) = b_{2}^{(2)}(M) \geq 0,\] whence $q(F) \geq 0$. Assume then, by contradiction, that equality holds for some manifold $M$; by \cite[Theorem 6]{E97}  the only obstruction for $M$ to be an Eilenberg--Maclane space $K(F,1)$ is $H^{2}(F,\Z[F])$. Now for Thompson's group $F$ all cohomology groups $H^{*}(F,\Z[F])$ vanish (\cite[Theorem 13.11.1]{Ge08}), so the obstruction vanishes; but in that case $F$ would be a Poincar\'e duality group of dimension $4$, hence satisfy $H^{4}(F,\Z[F]) = \Z$, that is false by the above. \end{proof}

We observe that the result above entails that the deficiency of $F$ is actually equal to zero. However, as the homology of $F$ is known (see e.g. \cite{Br06}), this follows also from Morse inequality $\mbox{def}(F) \leq b_1(F) - b_{2}(F) = 0$ and the existence of the presentation of (\ref{eq:pres}) of deficiency $0$.

\subsection*{Acknowledgment} We thank the referee for helpful comments.



\end{document}